     \documentclass[11pt,letterpaper]{amsart}
\usepackage{amsfonts}
\usepackage[all]{xy}
\usepackage{amsmath,amssymb,amsthm,latexsym,xcolor}
\numberwithin{equation}{section}
\usepackage{setspace}
\usepackage{caption}
\usepackage{bbm}

\usepackage{graphicx}
\usepackage{yhmath}

\newtheorem{theorem}{Theorem}[section]
\newtheorem{proposition}[theorem]{Proposition}

\newtheorem{lemma}[theorem]{Lemma}

\theoremstyle{remark}
\newtheorem{remark}[theorem]{Remark}
\newtheorem{example}[theorem]{\bf Example}
\newcommand{\h}{{\bf H}}
\newcommand{\R}{\mathbb{R}}
\newcommand{\cl}{\mathbb{S}^1(\frac{\sqrt2}{2})}
\newcommand{\s}{\mathbb{S}^3}
\newcommand{\stwo}{\mathbb{S}^2}
\newcommand{\sone}{\mathbb{S}^1}
\newcommand{\sn}{\mathbb{S}^n}

\newcommand{\CC}{R_{m,k}}

\newcommand{\hCC}{\widehat{R}_{m,k}}
\newcommand{\Z}{\mathbb{Z}}
\newcommand{\V}{\mathbb{V}}
\newcommand{\DD}{\mathcal{D}}
\newcommand{\E}{\mathbb{E}}

\newcommand{\dd}{\delta}

\newcommand{\F}{\mathcal{F}}
\newcommand{\TT}{\mathcal{T}}

\newcommand{\tg}{\widetilde{G}_{m,k}}

\newcommand{\g}{{G}_{m,k}}
\newcommand{\hg}{\widehat{G}_{m,k}}

\newcommand{\area}{\operatorname{area}}
\newcommand{\blue}{\textcolor{blue}}

\newcommand{\la}{{\langle}}
\newcommand{\ra}{\rangle} 
\newcommand{\RP}{\mathcal{R}_P}
\newcommand{\RQ}{\mathcal{R}_Q}

\newcommand{\hp}{\pi}
 
\begin{document}
\title [The Willmore problem for surfaces with symmetry  ]{\bf{The Willmore problem for surfaces  with symmetry}}  

\author
[Rob Kusner, Ying L\"{u} \& Peng Wang]{Rob Kusner, Ying L\"{u} \& Peng Wang}
\address {Department of Mathematics, University of Massachusetts, Amherst, MA 01003, USA} \email{profkusner@gmail.com, kusner@umass.edu}
\address{School of Mathematical Sciences, Xiamen University, Xiamen, 361005, P. R. China} \email{lueying@xmu.edu.cn}
\address{School of Mathematics and Statistics, Key Laboratory of Analytical Mathematics and Applications (Ministry of Education), Fujian Key Laboratory of Analytical Mathematics and Applications (FJKLAMA), Fujian Normal University, Fuzhou 350117, P. R. China} \email{netwangpeng@163.com, pengwang@fjnu.edu.cn}

\date{\today}
\begin{abstract}
The Willmore Problem seeks closed surfaces in $\s\subset\R^4$ of a given topological type minimizing the squared-mean-curvature energy $W = \int |\h_{\R^4}|^2  = \area + \int H_{\s}^2$.
The longstanding Willmore Conjecture that the Clifford torus minimizes $W$ among genus-$1$ surfaces is now a theorem of Marques and Neves \cite{MN},
but the general conjecture \cite{Kusner} that Lawson's \cite{Lawson} minimal surface $\xi_{g,1}\subset\s$ minimizes $W$ among surfaces of genus $g>1$ remains open. Here we prove this conjecture under the additional assumption that the competitor surfaces $M\subset\s$ share the ambient symmetries $\widehat{G}_{g,1}$ of $\xi_{g,1}$.  In fact, we show each Lawson surface $\xi_{m,k}$ satisfies the corresponding $W$-minimizing property under a smaller symmetry group $\widetilde{G}_{m,k}=\widehat{G}_{m,k}\cap SO(4)$. We also describe a genus 2 example where known methods do not ensure the existence of a $W$-minimizer among surfaces with its symmetry.
\end{abstract}

 
\maketitle

\section{Introduction}

The Willmore Problem of finding closed surfaces $M$ with given topological type in the unit 3-sphere $\s\subset\R^4$
which minimize the squared-mean-curvature energy
\[W(M) = \int_M |\h_{\R^4}|^2  = \area(M)  + \int_M H_{\s}^2\]
has led to much progress in geometric analysis
(see, for example, \cite{Brendle, ElSoufi1,KusnerMSRI,Li-y,Marques0,MN,MN2,Montiel,Neves}).
Willmore settled the problem for surfaces of genus $0$, showing the $W$-minimizer must be a round sphere, and Bryant \cite{Bryant1984} showed all $W$-critical surfaces of genus $0$ arise from 
compactifying complete minimal surfaces in $\R^3$ with embedded planar ends; except for the round sphere, these are all $W$-unstable \cite{HirschJDG, HKM, Michelat-CAG}. The problem for genus $1$ is the subject of the famous Willmore Conjecture, proven in the landmark paper of Marques and Neves \cite{MN} to be M\"obius equivalent to the minimal Clifford torus $T=\cl\times\cl\subset\s$ with $W=\area(T)=2\pi^2$.

In the 1980s, motivated by work of Weiner \cite{Weiner} and Li-Yau \cite{Li-y}, one of us \cite{Kusner} observed that the embedded minimal surfaces $\xi_{m,k}\subset\s$
of genus~$mk$ constructed by Lawson \cite{Lawson} satisfy $\area(\xi_{m,k})<4\pi(\min(m,k)+1)$, and in particular $W(\xi_{g,1})=\area(\xi_{g,1})<8\pi$ for the
simplest surface $\xi_{g,1}$ of genus~$g$ (the Clifford torus $T=\xi_{1,1}$). 
The latter estimate implies the infimum $\beta_g$ of $W(M)$ among surfaces $M$ of genus $g\ge1$ satisfies $\beta_g<8\pi$, and
since we now \cite{MN} know $\beta_g\ge\beta_1=2\pi^2>6\pi$, this establishes the Douglas-type condition ($\beta_{e}+\beta_{f}>\beta_{e+f}+4\pi$ for any $e,f\in \Z^+$) that Simon \cite{Simon93} needed to prove
a $W$-minimizing surface of each genus $g\ge1$ exists and is embedded \cite{Li-y}. 
An earlier gluing construction \cite{BK,Kusner96} also established
this Douglas-type condition directly.

Based on a preliminary study (inspired by \cite{Weiner}) of the $W$-stability properties of $\xi_{g,1}$ for small genus, and on limiting properties of the $\xi_{g,1}$ family for large genus
\cite{KusnerDiss}, it was conjectured \cite{Kusner} that for closed surfaces of genus $g$ in $\s$, the Willmore energy is greater or equal to $\area(\xi_{g,1})$ and equality holds if and only if the surface is congruent via the M\"obius group to $\xi_{g,1}$ --- a genus-$g$ version of the Willmore Conjecture.  But until recently there has been little progress towards identifying the solution to the genus-$g$ Willmore Problem as the Lawson surface $\xi_{g,1}$, aside from experimental evidence (see \cite{HKS} or  \cite{KusnerMSRI}), and a heuristic observation \cite{KusnerDiss} that if the $W$-minimizer were realized by a smallest area minimal surface of that genus, then the large-genus varifold limit \cite{KLS} for the $W$-minimizers would be supported  on the union of 2 round spheres meeting orthogonally (like the $\xi_{g,1}$-limit) rather than a round sphere with multiplicity 2.


A greater understanding (see \cite{Brendle2} for a succinct survey) of the analytic properties for the Lawson minimal surfaces
$\xi_{m,k}$ has emerged recently.  
\begin{itemize}
\item
Choe and Soret showed \cite{choesoret} that the first Laplace eigenvalue of $\xi_{m,k}$ is $2$, confirming Yau's conjecture for the Lawson minimal surfaces $\xi_{m,k}$. We modified \cite{KW2} their argument to show the first Laplace eigenspace of $\xi_{m,k}$ coincides with the $4$-dimensional space spanned by the coordinate functions. It follows (using the idea of conformal area introduced in \cite{Li-y} and sharpened by Montiel-Ros \cite{Montiel}) that $\xi_{m,k}$ is the unique $W$-minimizer (up to congruence in the M\"obius group) among all conformal immersions of the underlying Riemann surface $[\xi_{m,k}]$ into $\sn$ for any $n\geq3$.
\vspace{2mm}
\item
Kapouleas and Wiygul showed \cite{KaW} the area-index of $\xi_{g,1}$ is exactly $2g+3$, using its symmetries to decompose the low eigenspaces of its area Jacobi operator $\mathcal{L}$.
One can go further to show \cite{KW2} that $\mathcal{L}$ has the critical spectral gap \cite{Weiner} between $-2$ and $0$, and thus $\xi_{g,1}$ is $W$-stable in $\s$; moreover, if we regard $\xi_{g,1}\subset\sn$ by linearly embedding $\s\subset \sn$, its $W$-stability persists --- indeed, they are in fact local $W$-minimizers in $\sn$. Kapouleas and Wiygul   \cite{KaW2} also used Schoen's version \cite{SchoenJDG83} of the Alexandrov moving plane method to study  symmetries and uniqueness of $\xi_{m,k}$.
\end{itemize}
\noindent
Although the above results make essential use of the Lawson surface symmetries, it is unclear whether the direct method can be used to prove the existence of a $W$-minimizer among surfaces with a given symmetry group and topology: indeed, this family of genus 2 surfaces (see Figure~\ref{Fig:SdC}) embedded in~$\s$ is invariant under an action of a subgroup $\widehat{K}<O(4)$ of order~$8$ (extending its Klein halfturn subgroup $K<SO(4)$ of order $4$)
but the methods of \cite{BK,Kusner96} do not establish the Douglas-type condition, so a $W$-minimizing sequence of these $\widehat{K}$-symmetric surfaces might degenerate to a round 2-sphere with ``infinitesimal handles" at the left and right.

\begin{figure}[htb] 
\includegraphics[width=.85\textwidth]{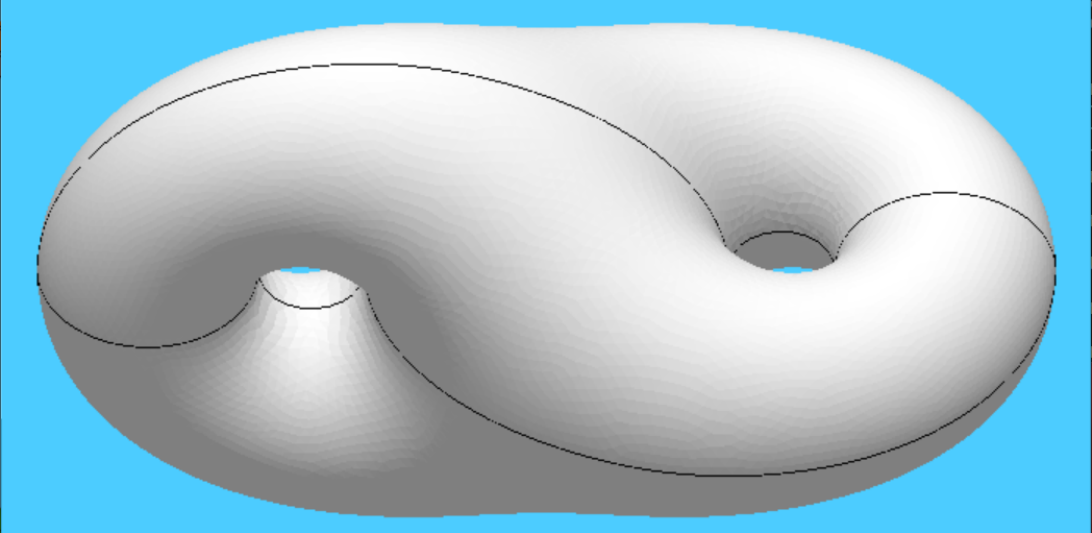}
\caption{Does a $W$-minimizing sequence of genus 2 surfaces with $\widehat{K}$-symmetry converge to a $W$-minimizer of
      lower genus? [Image courtesy of Ken Brakke]}\label{Fig:SdC}
\end{figure}

All of the above motivates the basic question we begin to address here:
\begin{center}
\emph{For finite subgroups of $SO(4)$ or $O(4)$ acting on an embedded surface in $\s$ is there a symmetric $W$-minimizing surface of the same topology?} 
\end{center}




\vspace{1mm}
\noindent 
The main purpose of this paper is to explore when both existence and uniqueness {\em can} be established for solutions to this {\em Willmore Problem for surfaces with symmetry},
and we prove that the Lawson minimal surface $\xi_{m,k}$ is the unique $W$-minimizer among embedded surfaces of the same genus
which are invariant with respect to any of the following order-$4(m+1)(k+1)$ subgroups of its full symmetry group.

\begin{theorem}\label{th}
Let $\tg<SO(4)$ be the halfturn-symmetry group of order $4(m+1)(k+1)$ for the embedded Lawson minimal surface $\xi_{m,k}\subset\s$ of genus $mk$,
and let $G$ be any group conjugate to $\widetilde{G}_{m,k}$ in the M\"obius group of conformal transformations of $\s$. If $M\subset\s$ is a closed embedded surface of genus $mk$ which is symmetric under $G$, then
\[W(M)\geq W(\xi_{m,k})=\area(\xi_{m,k})\geq W(\xi_{mk,1})=\area(\xi_{mk,1}).\]
The first equality holds if and only if $M$ is M\"obius-congruent to $\xi_{m,k}$, and the second equality holds if and only if $k=1$.  

The same holds if $\tg$ is replaced by the order-$4(m+1)(k+1)$ groups ${G}^P_{m,k}$ or
${G}^Q_{m,k}<O(4)$ generated by a single reflection symmetry together with the Schwarz-halfturn symmetry group $G_{m,k}\vartriangleleft\tg$ of index $2$.
\end{theorem}

\noindent
Here and throughout the paper, we assume $m\geq k\geq 1$ to avoid trivial cases ($\xi_{m,0}$ is a great $\stwo\subset\s$).  We note that $\tg$ is the maximal rotation-symmetry group for $\xi_{m,k}$ if $m\ne k$, and $G_{m,k}$ is the maximal common subgroup of $\tg$, ${G}^P_{m,k}$ and ${G}^Q_{m,k}$, with an index-$2$ maximal abelian subgroup $\CC\vartriangleleft G_{m,k}$ (see \eqref{eq-groups} for definitions of these groups).

\vspace{4mm}
\noindent
\textbf{Plan of the paper.}
Section 2 describes $\xi_{m,k}$ and its symmetry groups via the 3-orbifold $\s/\CC\cong\s$ containing the 2-orbifold $\xi_{m,k}/\CC\cong\stwo$. Two observations simplify our description of the group actions:  the chain of index-2 normal subgroups $\CC\vartriangleleft G_{m,k}\vartriangleleft\tg$ induces a tower of normal coverings $\s/\CC\to\s/G_{m,k}\to\s/\tg$, and $\CC\vartriangleleft\tg$ so the order-$4$ quotient group $\tg/\CC$ acts by deck transformations of the degree-4 orbifold covering $\s/\CC\to\s/\tg$; this is used in Section 3 to determine the topology and geometry of the 2-orbifold $M/\CC$ in $\s/\CC$ for any $G_{m,k}$-symmetric surface $M$ of genus $g$ satisfying $1<g\leq mk+k$, and to show $g=mk$.  Section~4 finishes the proof of Theorem \ref{th}.  Section~5 concludes with a topological characterization of $G_{m,k}$-symmetric surfaces in~$\s$.

\section{On the symmetries of  the Lawson minimal surface $\xi_{m,k}$}

In this section we recall the construction of the Lawson surface $\xi_{m,k}\subset\s$, and then discuss its symmetry groups, particularly the maximal abelian subgroup $\CC$.
We introduce the surface orbifold $\xi_{m,k}/\CC$ and observe it is an embedded 2-sphere in the 3-orbifold $\s/\CC$ that is itself homeomorphic to $\s$; this orbifold perspective \cite{Conway, Thurston} is essential here.
\subsection{Construction of $\xi_{m,k}$}

Lawson \cite{Lawson} constructed the minimal surface $\xi_{m,k}\subset\s\subset\mathbb{R}^4=\mathbb{C}^2$ of genus $mk$ by
considering the two orthogonal great circles
\begin{equation}\label{eq:theta}
	\gamma=\{P= (0, e^{i\theta} )|\, \theta\in[-\pi,\pi]\},~ \gamma^\perp=\{Q= (e^{i\theta^\perp},0)|\,\theta^\perp\in[-\pi,\pi]\}
\end{equation}
 in $\s$, placing equally-spaced points
\begin{equation*}
P_j = \left(0,e^{i\frac{j\pi}{k+1}}\right),
\quad Q_l = \left(e^{i\frac{l\pi}{m+1}}, 0\right),\quad
j\in \Z_{2k+2},\ l\in \Z_{2m+2}
\end{equation*}
 on  $\gamma$  and $\gamma^\perp$ respectively, such that $\{P_j\}$ decomposes $\gamma$ into $2k+2$ arcs, and $\{Q_l\}$ decomposes $\gamma^\perp$ into $2m+2$ arcs. Set
 \begin{equation}\label{eq-Gamma0}
  \begin{split}
    \Sigma_{P_j}&=   \hbox{the geodesic sphere in $\s$ determined by $P_j$ and $\gamma^\perp$}, \\
    \Sigma_{Q_l}&=   \hbox{the geodesic sphere in $\s$ determined by by $Q_l$ and $\gamma$}, \\
   \gamma_{j,l}&=   \hbox{the geodesic circle containing $P_{j}$ and $Q_l$.} \\
  \end{split}
\end{equation}
So $\sharp\{ \Sigma_{P_j}\}=k+1$,  $\sharp\{ \Sigma_{Q_l}\}=m+1$ and  $\sharp\{\gamma_{j,l}\}=(k+1)(m+1)$.  Moreover, let $P_jQ_l$, $P_{j}P_{j+1}$ and $Q_lQ_{l+1}$ denote the shortest geodesic arcs joining the respective points. The collection of great spheres
\[\left\{\Sigma_{P_j},\Sigma_{Q_l}|\,0\leq j\leq k, 0\leq l\leq m\right\}\] decomposes $\s$ into a collection $\{\TT_{j,l}\}$ of $4(m+1)(k+1)$  tetrahedral tiles, where the tile $\TT_{j,l}$ is the convex hull of $\{P_{j},P_{j+1},Q_{l},Q_{l+1}\}$ in $\s$.

Let $\Gamma_{j,l}$ be the geodesic quadrilateral $P_{j}Q_{l}P_{j+1}Q_{l+1}$ on the boundary of   $\TT_{j,l}\subset\s$. By \cite{Lawson}, among all surfaces in $\TT_{j,l}$ with boundary $\Gamma_{j,l}$, there exists a unique minimal disk $\delta_{j,l}$ with least area (see \cite{KaW2} for more on uniqueness).
Repeated Schwarz halfturns about the collection of great circles $\{\gamma_{j,l}\}$ applied to this disk generates an embedded closed minimal surface in $\s$,
extending smoothly across the edges of $\Gamma_{j,l}$, which is exactly
$\xi_{m,k}$ when $j+l$ is even (or a rotated {\em odd companion} surface $\xi^o_{m,k}$   when $j+l$ is odd; see Example 2.2 and also \cite{Brendle2,choesoret,KaW,Lawson}).\vspace{2mm}

To describe the symmetries of $\xi_{m,k}$ we also consider (as in \cite{KaW, KaW2}) the dual tiling $\{\TT^*_{j,l}\}$ of $\mathbb{S}^3$:
the tile $\TT^*_{j,l}$ here is the convex hull of $\{P^*_{j},P^*_{j+1},Q^*_{l},Q^*_{l+1}\}$ in $\s$, with 
\begin{equation*}\label{eq-pq*}
\begin{split}
&P_j^* = \left(0,e^{i \frac{(2j+1)\pi}{2(k+1)}} \right),~~ Q_l^* = \left(e^{i \frac{(2l+1)\pi}{2(m+1)}}, 0\right)   
\end{split}
\end{equation*}
 the midpoints of the edges $P_jP_{j+1}$ and $Q_lQ_{l+1}$, respectively.
So $\{P_j^*\}$ divides $\gamma$ into $2k+2$ arcs, and $\{Q_l^*\}$ divides $\gamma^\perp$ into $2m+2$ arcs, as $\{P_j\}$ and $\{Q_l\}$ do.
We similarly define the dual great circle $\gamma_{j,l}^*$, the dual great 2-spheres $\Sigma_{P_j^*}$ and $\Sigma_{Q_l^*}$, and the dual boundary quadrilateral $\Gamma^*_{j,l}$ with its unique Plateau disk $\delta^*_{j,l}$.  The dual Lawson minimal surface $\xi_{m,k}^*$ is congruent to $\xi_{m,k}$ by an ambient rotation $\RP^q\RQ^q$ taking $\{P_j^*, Q_l^*\}$ to $\{P_j, Q_l\}$. Here $\RP^q$ and $\RQ^q$ are rotations about $\gamma$ and $\gamma^\perp$ by angles $\frac{\pi}{2(m+1)}$ and $\frac{\pi}{2(k+1)}$, respectively:
\[	\RP^q(z_1,z_2)=(z_1e^{i\frac{\pi}{2(m+1)}},z_2),\quad
\RQ^q(z_1,z_2)=(z_1,z_2 e^{i\frac{\pi}{2(k+1)}}).
\]

\subsection{Symmetries of $\xi_{m,k}$}
For any great circle $\tilde\gamma$ or great 2-sphere $\Sigma$ in $\s$, we denote by $\underline{\tilde\gamma}$ the halfturn fixing $\tilde\gamma$ and by $\underline{\Sigma}$
the reflection fixing $\Sigma$. For example,
\begin{equation*} 
\begin{split}
&    \underline{\Sigma_{P_j}}(z_1,z_2)=({z_1},\overline{z_2}e^{i\frac{2j\pi}{k+1}}),\\
&\underline{\Sigma_{Q_l}}(z_1,z_2)=(\overline{z_1}e^{i\frac{2l\pi}{m+1}},z_2),\\
&\underline{\gamma_{j,l}}=\underline{\Sigma_{Q_l}}\cdot\underline{\Sigma_{P_j}}.\end{split}
\end{equation*}

\noindent
By its construction, the Lawson minimal surface $\xi_{m,k}$ is symmetric with respect to every halfturn $\underline{\gamma_{j,l}}$, and by uniqueness of the Plateau disks in its construction, the reflections $\underline{\Sigma_{P^*_j}}$ and $\underline{\Sigma_{Q^*_l}}$ are also symmetries. Thus $\xi_{m,k}$ is also symmetric under the halfturns
$\underline{\gamma^*_{j,l}}=\underline{\Sigma_{Q^*_l}}\cdot\underline{\Sigma_{P^*_j}}$, as well as under $\RP$ and $\RQ$, where $\RP$ is the rotation about $\gamma$ by angle $\frac{2\pi}{m+1}$ and  $\RQ$ is the rotation about $\gamma^{\perp}$ by angle $\frac{2\pi}{k+1}$, that is, by these two elements  
\[	\RP(z_1,z_2)=(z_1e^{i\frac{2\pi}{m+1}},z_2),\quad
\RQ(z_1,z_2)=(z_1,z_2 e^{i\frac{2\pi}{k+1}})
\]
in the maximal torus $SO(2)\times SO(2)=U(1)\times U(1)<U(2) < SO(4)$ preserving $\gamma\cup\gamma^{\perp}$ setwise.

One can readily check that $\xi_{m,k}$ is symmetric with respect to all the following groups:
\begin{equation}\label{eq-groups}
\left\{
\begin{aligned}
&\CC=\la \RP,\RQ\ra < SO(4),\\
&\CC^P=\CC\rtimes\la\underline{\Sigma_{P_0^*}}\ra < O(4),\\
&\CC^Q=\CC\rtimes\la\underline{\Sigma_{Q_0^*}}\ra < O(4),\\
&\hCC=\CC\rtimes\la\underline{\Sigma_{P_0^*}}, \underline{\Sigma_{Q_0^*}}\ra =\la\underline{\Sigma_{P_0^*}}, \underline{\Sigma_{P_1^*}}, \underline{\Sigma_{Q_0^*}}, \underline{\Sigma_{Q_1^*}}\ra< O(4)\\
&G_{m,k}=\CC\rtimes \la\underline{\gamma_{1,1}}\ra=\la \mathcal{R}_P,\mathcal{R}_Q,\underline{\gamma_{1,1}}\ra < SO(4), \\
&G^*_{m,k}=\CC\rtimes \la\underline{\gamma^*_{1,1}}\ra=\la \mathcal{R}_P,\mathcal{R}_Q,\underline{\gamma^*_{1,1}}\ra < SO(4),\\
&\check{G}_{m,k}=G^*_{m,k}\rtimes\la\underline{\Sigma_{P_0^*}}\ra=G^*_{m,k}\rtimes\la\underline{\Sigma_{Q_0^*}}\ra=\hCC <O(4),\\
 &\tg=G_{m,k}\rtimes \la\underline{\gamma^*_{1,1}}\ra=\la \mathcal{R}_P,\mathcal{R}_Q, \underline{\gamma_{1,1}}, \underline{\gamma^*_{1,1}}\ra <SO(4),\\
&{G}_{m,k}^{P}=G_{m,k}\rtimes\la\underline{\Sigma_{P_0^*}}\ra < O(4),\\
 &{G}_{m,k}^{Q}=G_{m,k}\rtimes \la\underline{\Sigma_{Q_0^*}}\ra <O(4),\\
 &\hg=\tg\rtimes\la\underline{\Sigma_{P_0^*}}\ra=\tg\rtimes\la\underline{\Sigma_{Q_0^*}}\ra < O(4).
\end{aligned}
\right.
\end{equation}
 Note that the subgroup $R_{m,k}<SO(4)$ depends on $\gamma$ and $\gamma^{\perp}$ --- that is, on the choice of maximal torus $SO(2)\times SO(2)$ --- but it is independent of the points $\{P_j,Q_l\}=\xi_{m,k}\cap(\gamma\cup\gamma^{\perp})$. In fact, all the other groups contain involutions fixing some points on  $\gamma$ or $\gamma^\perp$, and though some are isomorphic as groups, with the exception $\check{G}_{m,k}=\hCC$, they act differently on $\xi_{m,k}\subset \s$.

The relations between the groups in \eqref{eq-groups} are as follows:
\begin{equation}\label{eq-G-seq}
\begin{aligned}
 \xymatrix{
       &   {G}_{m,k}^{P},{G}_{m,k}^{Q}\ar[dl]^{2:1}_{\vartriangleright} &  \ar[l]^{\quad2:1}_{\quad\vartriangleright} G_{m,k} \ar[dl]^{2:1}_{\vartriangleright}&   \\
\hg &   \tg  \ar[l]^{2:1}_{\vartriangleright}  & &\ar[ll]^{4:1}_{\vartriangleright} \CC   \ar[ul]^{2:1}_{\vartriangleright}\ar[dl]^{2:1}_{\vartriangleright}   \\
         &\check{G}_{m,k} \ar[ul]^{2:1}_{\vartriangleright} & G^*_{m,k}   \ar[l]^{2:1}_{\vartriangleright}  \ar[ul]^{2:1}_{\vartriangleright}    & \\
}
\end{aligned}
\end{equation}
Obviously each index-2 subgroup is normal in its supergroup.
To see $\CC$ is normal in both $\tg$ and $\hg$, observe that the halfturns $\underline{\gamma_{j,l}^{\ } }, \underline{\gamma^*_{j,l}}$
and the reflections $\underline{\Sigma_{P^*_j}}$, $\underline{\Sigma_{Q^*_l}}$ preserve or reverse the cyclic orders of $\{P_j\}\subset\gamma$ or $\{Q_l\}\subset\gamma^{\perp}$, but conjugation by any of these preserves both orders, so conjugation in both $\tg$ and $\hg$ preserves $\CC$ (see diagram \eqref{eq-G-seq-Rmk} below). 

For $m=k$ an extra symmetry  $\underline{\varepsilon}(z_1,z_2):=(z_2,z_1)$ in $ U(2) < SO(4)$ exchanges $\gamma$ and $\gamma^\perp$, generating an extension $\overline{G}_{m,m}:=\widehat{G}_{m,m}\rtimes \la\underline{\varepsilon}\ra<O(4)$ of  degree $2$. We can now state the following basic observation of Kapouleas and Wiygul.

\begin{lemma}{\rm(Lemma 3.10 of \cite{KaW2})} The full symmetry group\footnote{Here $\hg$ and $\tg$ are the groups $\mathcal{G}$  and  $\mathcal{G}^{\mathbf{C}}$ of \cite{KaW2}, but they shift Lawson's \cite{Lawson} (and our) indices $m,k$ by 1.} of $\xi_{m,k}$ is $\hg$ when $m>k\geq1$, and the full symmetry group of $\xi_{m,m}$ is  $\overline{G}_{m,m}$  when $m=k>1$.
\end{lemma}

\subsection{The spherical orbifold $\s/\CC$}

Here we discuss the geometry and topology of the spherical 3-orbifold $\s/\CC$ and its symmetries induced by halfturns or reflections of $\s$,
then study the image $\pi(\xi_{m,k})=\xi_{m,k}/\CC$ of the Lawson minimal surface $\xi_{m,k}$ in $\pi(\s)=\s/\CC$ under the natural degree-$(m+1)(k+1)$ orbifold covering map
\begin{equation}
    \label{eq-pi}
\hp:\s\to \s/\CC
\end{equation}
with singular set $\gamma\cup\gamma^{\perp}$ (see subsection \ref{sec-singular}).  The orbifold $\pi(\s)=\s/\CC$ inherits a locally spherical geometry away from the singular set where it is spherically conical and has isometry group \[\text{Iso}(\s/\CC)=\left\{
\begin{split}
   &(O(2)\times O(2))/\CC \hbox{ when } m\neq k,\\
   &((O(2)\times O(2))\rtimes\la\underline{\varepsilon}\ra)/\CC \hbox{ when } m=k.\\
\end{split}\right.\]

\subsubsection{Topology of the 3-orbifold $\s/\CC$}
Recall that the 3-sphere is homeomorphic to the join of two circles: a \emph{geometric join} realizing the round sphere \[\s=\gamma\bowtie\gamma^\perp=\cup_{P\in\gamma,Q\in\gamma^{\perp}} \,P\bowtie\!Q\] results by letting $P\bowtie Q$ denote the shortest geodesic segment $\{P\cos t+Q\sin t \,\, | \,\,t\in [0,\frac{\pi}{2}]\}$ between $P$ and $Q$. This geometric join respects the $\CC$ action on $\s$.

Each of the quotients \[\mathbb{S}_1:=\pi(\gamma)=\gamma/\CC=\gamma/\RQ \quad \hbox{ and }\quad  \mathbb{S}_2:=\pi(\gamma^\perp)=\gamma^\perp/\CC=\gamma^\perp/\RP\]
is still homeomorphic to a circle, so $\s/\CC=\mathbb{S}_1\bowtie\mathbb{S}_2$ is still homeomorphic to $\s$.

\subsubsection{Flip-symmetries of $\s/\CC$}

Decompose each circle $\mathbb{S}_j$ into two vertices $\V_j^+$, $\V_j^-$ and two edges $\E_j^+$, $\E_j^-$ as in the figure below, with  $\V_j^{*+}$ and $\V_j^{*-}$ being the midpoints of $\E_j^+$ and of $\E_j^-$,
respectively, chosen such that
\begin{align*}
\hp(P_{2u})=\V_1^+,\quad \hp(P_{2u+1})=\V_1^-,&\quad \hp(P^*_{2u})=\V_1^{*+},\quad \hp(P_{2u+1}^*)=\V_1^{*-},\\
\hp(Q_{2u})=\V_2^+, \quad\hp(Q_{2u+1})=\V_2^-,&\quad\hp(Q_{2u}^*)=\V_2^{*+}, \quad\hp(Q_{2u+1}^*)=\V_2^{*-}.
\end{align*}
The resulting marked circle $\mathbb{S}_j$ has several symmetries: $\F_j^v$ is the \emph{vertical flip} with $\F_j^v(\V_j^{\pm})=\V_j^{\pm}$ and $\F_j^v(\E_j^{\pm})=\E_j^{\mp}$; similarly $\F_j^h$ is the  \emph{horizontal flip} with $\F_j^h(\V_j^{\pm})=\V_j^{\mp}$ and $\F_j^h(\E_j^{\pm})=\E_j^{\pm}$.
\begin{figure}[htbp]
	\includegraphics[width=0.4\textwidth]{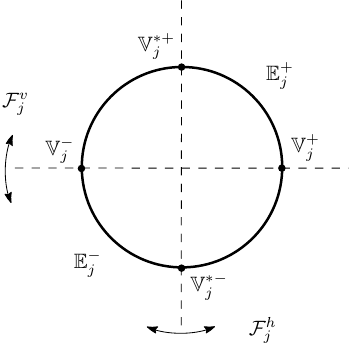} 
\end{figure}

The spherical reflections $\underline{\Sigma_{P_j}}$ and $\underline{\Sigma_{Q_l}}$ induce $\F_1^v$ and $\F_2^v$, while $\underline{\Sigma_{P^*_j}}$ and $\underline{\Sigma_{Q^*_l}}$ induce $\F_1^h$ and $\F_2^h$, respectively.  These flips generate symmetry groups of  deck transformations acting on orbifold coverings  $\s/\CC\to\s/G$ with $G\leqslant\hg$ from \eqref{eq-groups} that are related as follows:
\begin{equation}\label{eq-G-seq-Rmk}
\begin{aligned}
\xymatrix{
	     &   \langle\F_1^v\cdot\F_2^v,\F_j^h \rangle\ar[dl]^{2:1}_{\vartriangleright} &  \ar[l]^{2:1}_{\vartriangleright} \langle\F_1^v\cdot\F_2^v \rangle \ar[dl]^{2:1}_{\vartriangleright}&   \\
	\langle\F_1^v\cdot\F_2^v,\F_1^h,\F_2^h \rangle &   \langle\F_1^v\cdot\F_2^v,\F_1^h\cdot\F_2^h \rangle  \ar[l]^{2:1}_{\vartriangleright}  &  & \ar[ll]^{4:1}_{\vartriangleright}  \mathbbm{1}   \ar[ul]^{2:1}_{\vartriangleright}\ar[dl]^{2:1}_{\vartriangleright}   \\
	&\langle\F_1^h,\F_2^h \rangle \ar[ul]^{2:1}_{\vartriangleright} & \langle\F_1^h\cdot\F_2^h \rangle   \ar[l]^{2:1}_{\vartriangleright}  \ar[ul]^{2:1}_{\vartriangleright}   & }
 \end{aligned}
\end{equation}
One easily checks all their nontrivial elements have order 2, so the groups in \eqref{eq-G-seq-Rmk} are abelian:
$\langle\F_1^v\cdot\F_2^v,\F_1^h,\F_2^h \rangle \cong \Z_2\oplus \Z_2\oplus \Z_2$
and $\langle\F_1^v\cdot\F_2^v,\F_j^h \rangle\cong \langle\F_1^v\cdot\F_2^v,\F_1^h\cdot\F_2^h \rangle \cong\langle\F_1^h,\F_2^h \rangle
\cong  \Z_2\oplus \Z_2.$  Indeed, the  groups in \eqref{eq-G-seq-Rmk} are essentially the quotients of those in \eqref{eq-G-seq} modulo $R_{m,k}$.

\subsubsection{A fundamental domain in $\s$ for $\CC$}\label{sec-singular}
A convenient fundamental domain for the action of $\CC$ on $\s$ is
	$\mathcal D:=\TT_{0,0}\cup\TT_{0,1}\cup\TT_{1,0}\cup\TT_{1,1}$ and thus the orbifold $\s/\CC$ can also be viewed as $\mathcal D$ with boundary faces identified (left to right and top to bottom, as in Figure \ref{Fig:S3C}):
\begin{figure}[htbp]
\includegraphics[width=0.45\textwidth]{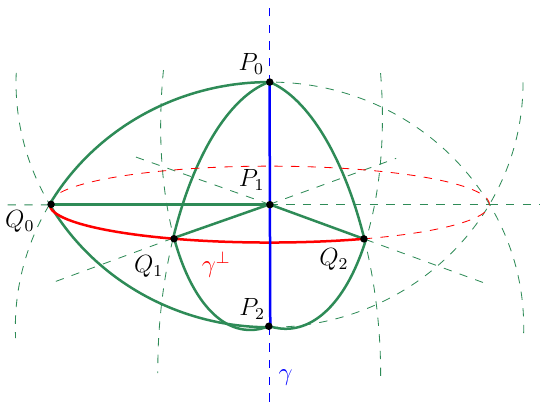}\hspace{8mm}
	\includegraphics[width=0.4\textwidth]{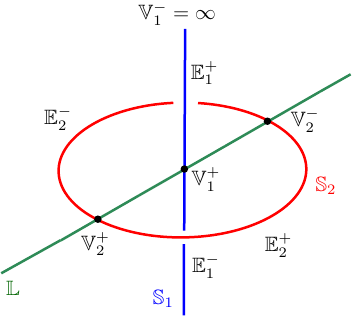}
	\caption{The fundamental domain $\mathcal D$ of $\CC$ and the orbifold $\s/R_{m,k}$}\label{Fig:S3C}
\end{figure}

The points on $\mathbb{S}_1\subset\s/\CC$ (the blue circle in Figure \ref{Fig:S3C}) are called {\em $m$-singular}, the points on $\mathbb{S}_2$ (the red circle) are {\em  $k$-singular}, and the remaining points in $\s/\CC$ are {\em regular}.  Define the {\em Lawson circle $\mathbb{L}$} (the green circle in Figure \ref{Fig:S3C}) to be the image under $\hp:\s\to\s/\CC$ of any halfturn great circle on $\xi_{m,k}$:
\begin{equation}\label{eq-L}
\mathbb{L}:=\hp(\cup\gamma_{j,l})=\cup\{\V_1^\pm\}\bowtie\{\V_2^{\pm}\}\subset\s/\CC.
\end{equation}
Note that any halfturn $\underline{\gamma_{j,l}}$ induces the halfturn $\F_1^v\cdot\F_2^v=\underline{\mathbb{L}}$ around $\mathbb{L}$ in $\s/\CC$. Similarly, we have $\F_1^h\cdot\F_2^h=\underline{\mathbb{L}^*}$ for the {\em dual Lawson circle} $\mathbb{L}^*=\hp(\cup\gamma_{j,l}^*)$.

\subsection{Locating $\pi(\xi_{m,k})$ and  $\pi(\xi^*_{m,k})$ in $\s/\CC$}
Here we locate in  $\s/\CC$ our four example surfaces: $\pi(\xi_{m,k})$ and $\pi(\xi^*_{m,k})$ along with their odd companions $\pi(\xi^o_{m,k})$ and $\pi(\xi^{*o}_{m,k})$.
We also determine where each surface meets the four circles $\mathbb{L}$, $\mathbb{L^*}$, $\mathbb{S}_1$ and $\mathbb{S}_2$.
\begin{example}
	A fundamental domain of $\xi_{m,k}$ under the action of $\CC$ is a disk bounded by arcs $P_jQ_l$ $ (0\leq j,l\leq2)$ lying in the even tiles $\TT_{0,0}\cup\TT_{1,1}$.
	Its projection $\hp(\xi_{m,k})$ is topologically a sphere (Figure \ref{Fig:lawson}).
	Its odd companion $\xi_{m,k}^o:=\underline{\Sigma_{P_0}}(\xi_{m,k})$  has a fundamental domain lying in the odd tiles $\TT_{0,1}\cup\TT_{1,0}$.
	 Note that $\xi_{m,k}$ meets $\gamma$, $\gamma^\perp$ and every $\gamma^*_{j,l}$ orthogonally, so its projection $\hp(\xi_{m,k})$ is topologically a sphere containing $\mathbb{L}$ meeting  each of the circles $\mathbb{S}_j$ and $\mathbb{L}^*$ in $\s/\CC$ transversally and exactly twice (see Figure \ref{Fig:lawson}; for $\xi_{m,k}^o$, we replace $\mathbb{A}^{\pm}$ by  $\mathcal{F}^v_1(\mathbb{A}^{\pm})$): \[\pi(\xi_{m,k})\supset\mathbb{L},~~ \pi(\xi_{m,k})\cap \mathbb{S}_j=\{\V_j^{+}, \V_j^{-}\},~~\pi(\xi_{m,k})\cap \mathbb{L}^*=:\{\mathbb{A}^{+}, \mathbb{A}^{-}\}.\]
Here we have set $\mathbb{A}^{+}$ to be the projection of the centers of the Plateau disks $\delta_{j,l}$ for $(j,l)$ both even, and $\mathbb{A}^{-}$ to be the projection of the centers of the Plateau disks $\delta_{j,l}$ for $(j,l)$ both odd.
	\begin{figure}[h]
		\includegraphics[width=.9\textwidth]{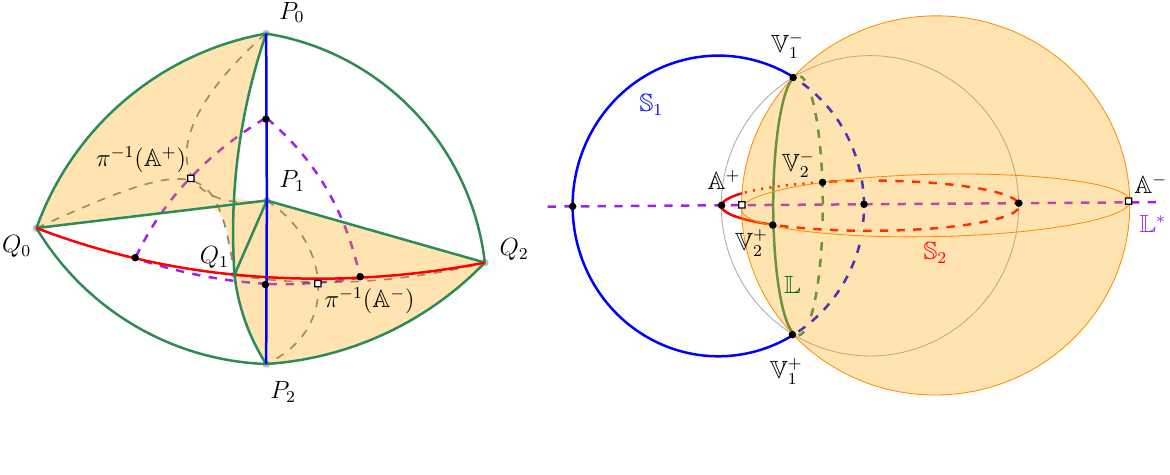}
		\caption{$\xi_{m,k}$ in $\DD$ and $\xi_{m,k}/\CC$ in $\s/\CC$}\label{Fig:lawson}
	\end{figure}

\end{example}

\begin{example} Similarly, the dual surface $\xi^*_{m,k}$ meets $\gamma$, $\gamma^\perp$ and every $\gamma_{j,l}$ orthogonally, so its projection $\hp(\xi^*_{m,k})$ is topologically a sphere containing $\mathbb{L}^*$ and meeting each of the circles $\mathbb{S}_j$ and $\mathbb{L}$ in $\s/\CC$ transversally and exactly twice (Figure \ref{Tmk}; for  $\xi_{m,k}^{*o}$ we  replace $\mathbb{A}^{*\pm}$ with  $\mathcal{F}^h_1(\mathbb{A}^{*\pm})$), where $\pi^{-1}(\mathbb{A}^{*\pm})$ is the set of centers for the dual Plateau disks $\{\delta^*_{j,l}\}$):
\[\pi(\xi^*_{m,k})\supset\mathbb{L^*},~~ \pi(\xi^*_{m,k})\cap \mathbb{S}_j=\{\V_j^{*+}, \V_j^{*-}\},~~\pi(\xi_{m,k})\cap \mathbb{L}=:\{ \mathbb{A}^{*+}, \mathbb{A}^{*-}\}.\]
	\begin{figure}[h]
		\includegraphics[width=.9\textwidth]{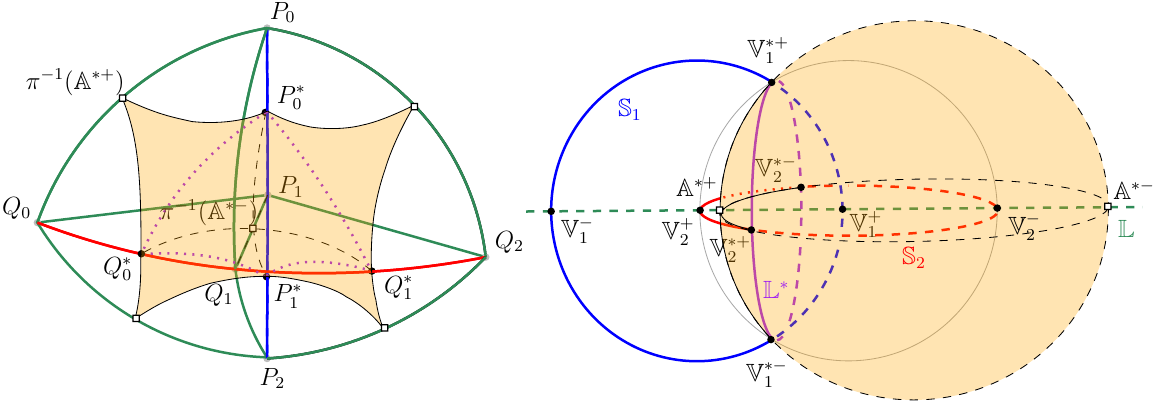}
		\caption{ $\xi^*_{m,k}$ in $\DD$ and  $\xi^*_{m,k}/\CC$ in $\s/\CC$}
		\label{Tmk}
	\end{figure}	
\end{example}

\begin{remark}
The illustrations above suggest how to locate a $G_{m,k}$-symmetric surface $M\subset \s$ of genus between $2$ and $mk$ in terms of the intersection pattern of $\pi(M)\subset\s/\CC$
with the four circles $\mathbb{L}$, $\mathbb{L^*}$, $\mathbb{S}_1$ and $\mathbb{S}_2$ in $\s/\CC$  (see Section 5).
\end{remark}

\section{Embedded $\g$-symmetric surfaces projected to $\s/\CC$}

In this section, we slightly relax the genus condition, and assume $M\subset \s$ is a closed, connected, embedded, $\g$-symmetric surface of genus $g$ with $1<g\leq mk+k$, where $\CC\vartriangleleft\g$ are the symmetry groups from (\ref{eq-groups}).
We determine the topology of the orbifold quotient surface $\hp(M)=M/\CC$ and its relative position with $\mathbb{S}_1$, $\mathbb{S}_2$ and $\mathbb{L}$ in $\hp(\s)=\s/\CC\cong\s$ as follows.
\begin{theorem}\label{thm-inter}
	Suppose $M\subset \s$ is a closed, connected, embedded, $\g$-symmetric surface of genus $g$ with $1<g\leq mk+k$.  Then $g=mk$ and $\hp(M)$ is homeomorphic to a $2$-sphere, and also $v_1=v_2=2$, where
 \begin{equation*}
v_j=\sharp \{\hp(M)\cap\mathbb{S}_j\},\quad j=1,2.
\end{equation*}
Moreover, either $\mathbb{L}\subset\hp(M)$, or $\hp(M)$ meets $\mathbb{L}$ transversally at exactly two points. 
\end{theorem}

\vspace{8mm}
We begin with several topological properties of $\hp(M)$:
\begin{lemma}\label{lem:piM}
The surface $\hp(M)\subset \s/\CC$ is connected, embedded and orientable; its boundary is nonempty if and only if $k=1$ and  $\pi(M)$ is tangent to  $\mathbb{S}_2$, in which case  $\partial (\pi(M))=\mathbb{S}_2$.
\end{lemma}
\begin{proof}
Since $M$ is connected and the covering projection $\hp:\s\to\s/\CC$ is continuous, $\hp(M)$ is connected.
Embeddedness of $\hp(M)=M/\CC$ in $\s/\CC$ is inherited from embeddedness of $M$ in $\s$.
Because $\CC<SO(4)$ acts by orientation-preserving isometries on $\s$, an orientation on $M$ induces an orientation on  $\hp(M)$.

The singular set $\mathbb{S}_1\cup\mathbb{S}_2$ of $\pi(\s)$ is closed, so any nonsingular point $p\in \pi(M)$ has a neighborhood $U\subset\pi(\s)$ disjoint from $\mathbb{S}_1\cup\mathbb{S}_2$.
Thus $U\cap\pi(M)$ is a neighborhood of $p$ in $\pi(M)$, and any boundary points of $\pi(M)$ are singular.
If $\pi(M)$ were tangent to $\mathbb{S}_1$ (or $\mathbb{S}_2$) at some $q$,
then $M$ would have at least $m+1$ sheets in $\s$ tangent to $\gamma$ (or at least $k+1$ sheets tangent to $\gamma^\perp$) at each point of $\pi^{-1}(q)$ due to the $R_{m,k}$-symmetry of $M$.
Since $m\geq k\geq 1$ and $mk>1$, the embeddedness of $M$ implies $k=1$ and $\pi^{-1}(q)\subset\gamma^\perp\subset M$;
thus $\mathbb{S}_2=\hp(\gamma^\perp)\subset\hp(M)$ must be the boundary of $\hp(M)$.
\end{proof}
\begin{lemma} \label{Le-PQ}
If $\pi(M)$ meets $\mathbb{S}_i$ transversally and contains any of  $\{\V_i^+, \V_i^-\}$, then $\mathbb{L}\subset \pi(M)$.
\end{lemma}

\begin{proof}
For instance, suppose $\V_i^+\in\pi(M)$.  The tangent cones $T_{\V_i^+}\pi(\s)$ and $T_{\V_i^+}\pi(M)$ are homeomorphic to $\R^3$ and $\R^2$ respectively, hence orientable.
The tangent map
\[\underline{\mathbb L}_{*\V_i^+}:T_{\V_i^+}\pi(\s)\rightarrow T_{\V_i^+}\pi(\s)\]of the involution $\underline{\mathbb{L}}:=\F_1^v\cdot\F_2^v$ is an orientation-preserving involution on $T_{\V_i^+}\pi(\s)$ with fixed-set~$T_{\V_i^+}\mathbb{L}$.
Since $\mathbb{S}_i$ meets $\pi(M)$ transversally at $\V_i^+$ and both are $\underline{\mathbb{L}}$-symmetric, both $T_{\V_i^+}\mathbb{S}_i$ and $T_{\V_i^+}\pi(M)$ are $\underline{\mathbb L}_*$-invariant in $T_{\V_i^+}\pi(\s)$, and they meet transversally in their common vertex.
Since $\underline{\mathbb L}_*$ reverses orientation on $T_{\V_i^+}\mathbb{S}_i$, it induces an orientation-reversing involution on $T_{\V_i^+}\pi(M)$, hence on the space of directions in $T_{\V_i^+}\pi(M)$, 
which is
homeomorphic to $\sone$.  But any orientation-reversing homeomorphism of $\sone$ has a fixed point, so
there exists nontrivial $u_0\in T_{\V_i^+}\pi(M)$ such that  $\underline{\mathbb L}_*u_0=\lambda u_0$ for some $\lambda>0$.
Since $\underline{\mathbb L}_*$ is an involution, $u_0=\underline{\mathbb L}_*\!\cdot\underline{\mathbb L}_*u_0=\lambda^2 u_0$, so $\lambda=1$ and $u_0\in T_{\V_i^+}\mathbb{L}$; that is, $\mathbb{L}$ is tangent to $\pi(M)$ at $\V_i^+$.
The $\underline{\mathbb{L}}$-symmetry of $\pi(M)$ and embeddedness of $M$  imply $\mathbb{L}\subset \pi(M)$, as in the proof of Lemma~\ref{lem:piM}.
\end{proof}

\begin{lemma}\label{lem:trans}
    If $\hp(M)$ is homeomorphic to a $2$-sphere, then $v_1 \geq 2$ and $v_2\geq 2$.
\end{lemma}
\begin{proof}
By the proof of Lemma \ref{lem:piM}, $M$ meets $\gamma$ and $\gamma^{\perp}$ at finitely many points.
Choose a path on $M$ connecting  $A\in M$ to $\mathcal{R}_P(A)$ avoiding $\gamma$. By repeated $\mathcal{R}_P$-action on this path we obtain a   loop $c$ on $M$ having nonzero linking number with $\gamma$, and $\hp(c)$ is a  loop on $\hp(M)$ having nonzero linking number with $\mathbb{S}_1$.
Since $\hp(M)$ is a $2$-sphere, $\hp(c)$ bounds a disk on $\hp(M)$ which must meet
$\mathbb{S}_1$.
The case for $\mathbb{S}_2$ is similar.
Obviously $\pi(M)$ is transversal to both of $\mathbb{S}_i$, so $v_i\geq 2$.
\end{proof}

\begin{proof}[Proof of Theorem \ref{thm-inter}]
We claim $\partial (\hp(M))=\emptyset$. Otherwise $\partial (\hp(M))=\mathbb{S}_2$ and $k=1$ by Lemma \ref{lem:piM}. So its manifold Euler number is $\chi(\hp(M))=1-2\hat{g}$ for some natural number~$\hat{g}$.
Since $\mathbb{S}_2$ is a closed curve, it contributes $0$ to  $\chi(\hp(M))$. Thus we only need to deal with the singular points $\hp(M)\cap\mathbb{S}_1$ to compute the Euler number of $M$  (as in  the Riemann-Hurwitz relation):
\[
\chi(M)=2-2g=2(m+1)(1-2\hat{g}-v_1)+2v_1=2+2m-4\hat{g}(m+1)-2mv_1.
\]
Since $g=2\hat{g}(m+1)+m(v_1-1)$, from
\begin{equation}\label{eq:nob}
    2\hat{g}(m+1)-m\leq g\leq mk+k=m+1
\end{equation}
we obtain $\hat{g}=0$. So $\pi(M)$ is a disk bounded by $\mathbb{S}_2$. Note that $\mathbb{S}_2$ is linked with $\mathbb{S}_1$, so $v_1$ is odd because $\pi(M)$ is always transversal to $\mathbb{S}_1$. From $g=m(v_1-1)>1$ we get $v_1\geq 3$ which yields $g\geq 2m>m+1$. 

Therefore, $\partial(\pi(M))=\emptyset$, which means $\pi(M)$ is also transversal to $\mathbb{S}_2$.  Both $v_1$ and $v_2$ are  even since $\pi(M)$ is embedded and separates $\s/\CC$.
Thus $\hp(M)$ must be a closed, embedded surface in $\hp(\s)$ with genus $\hat{g}$, so $\chi(M)$ can be computed (again as in  the Riemann-Hurwitz relation):
\[\chi(M)=2-2g=(m+1)(k+1)\chi_o(\pi(M)).\]
And we have \begin{equation*}
\begin{aligned}
\chi_o(\pi(M))&=2-2\hat{g}-v_1-v_2 +\frac{v_1}{m+1}+\frac{v_2}{k+1}\\
	&=\frac{2-2mk}{(m+1)(k+1)}+\frac{m(2-v_1)}{m+1}+\frac{k(2-v_2)}{k+1}-2\hat{g},
\end{aligned}
\end{equation*}
where $1/(m+1)$ and $1/(k+1)$ are respectively the weights of the $m$-singular and $k$-singular points in the orbifold $\s/\CC$,
and where $\chi_o(\pi(M))$ is the orbifold Euler number of $\hp(M)$. Hence
\[g=mk+(m+1)(k+1)\hat{g}-\frac{m(k+1)(2-v_1)+k(m+1)(2-v_2)}{2}.\]
Note that $v_1=v_2=0$ yields $g=1$ or $g>mk+k$ contrary to our assumption, and so $v_i\geq 2$ for some $i$. Since $g\leq mk+k$, we get \[2(m+1)(k+1)\hat{g}\leq m(k+1)(2-v_1)+k(m+1)(2-v_2)+2k,\] that is,
\begin{equation}\label{hat(g)}
\begin{split}
    \hat{g}&\leq  \frac{m(2-v_1)}{2m+2}+\frac{k(2-v_2)}{2k+2}+\frac{k}{(m+1)(k+1)}\\
    &\leq \frac{2m}{2m+2}+\frac{k}{(m+1)(k+1)}<1.
\end{split}
\end{equation}
So $\hat{g}=0$. By Lemma \ref{lem:trans}, $v_1=v_2=2$.

If $\pi(M)$ meets $\mathbb{S}_1$ at $\V_1^\pm$ then $\mathbb{L}\subset\hp(M)$. Otherwise, consider the sphere $\mathbb{S}_1\bowtie\mathbb{L}^*$, whose intersection with $\pi(M)$ can be assumed to be several simple loops with $\underline{\mathbb{L}}$-symmetry. Since $\mathbb{S}_1$ is transversal to $\pi(M)$, there is a unique simple loop $c^{*}\subset \pi(M)\cap(\mathbb{S}_1\bowtie\mathbb{L}^*)$ meeting $\mathbb{S}_1$ twice, and~$c^*$ must be $\underline{\mathbb{L}}$-symmetric. 
Since $\mathbb{S}_1\bowtie\mathbb{L}^*$ meets $\mathbb{L}$ twice,
$c^*$ bounds two $\underline{\mathbb{L}}$-symmetric disks on~$\mathbb{S}_1\bowtie\mathbb{L}^*$,
each meeting $\mathbb{L}$ once, and so $c^*$ has linking number $1$ with~$\mathbb{L}$.
Thus any other open disk bounded by $c^*$ must also meet $\mathbb{L}$,
in particular, each of the two $\underline{\mathbb{L}}$-symmetric embedded disks on $\pi(M)$ bounded by $c^*$. Let $D$ be one of these two open disks, and pick some $\V_0 \in D\cap\mathbb{L}$.  Clearly $\V_0$ is fixed by $\underline{\mathbb{L}}$,
and we claim it is the only fixed point in $D$.  To see this,
choose a simple path $c_1$ connecting $\V_0$ to $c^*$, avoiding any other possible points in $\mathbb{L}\cap D$.
The simple path $c_1\cup\underline{\mathbb{L}}(c_1)$ divides $D$ into two disjoint open disks $D_1$ and $D_2$.
Since $\partial D_1=\underline{\mathbb{L}}(\partial D_2)$ and $\partial D_1\not=\partial D_2$, we deduce $D_2=\underline{\mathbb{L}}(D_1)$,
which means there are no fixed points of  $\underline{\mathbb{L}}$ in $D_1$ or $D_2$. Hence $\mathbb{L}\cap D=\{\V_0\}$, which yields $\sharp(\mathbb{L}\cap\hp(M))=2$.
\end{proof}

\begin{remark}\label{Riemann-Hurwitz}
Another way to see this in the spirit of our orbifold Euler computation above:  $\underline{\mathbb{L}}$ is an orientation-preserving involution of $\pi(M)$ with no $\underline{\mathbb{L}}$-fixed loops; hence the quotient map $\pi(M)\cong\stwo \to \pi(M)/\underline{\mathbb{L}}$ is a degree-2 cover branched at the fixed points of $\underline{\mathbb{L}}$, and the Riemann-Hurwitz relation implies there are exactly 2 branch points.
\end{remark}

\section{Uniqueness of W-minimizing embedded symmetric surfaces}

In this section we prove Theorem \ref{th} by showing that $M$ is M\"{o}bius congruent to a surface  containing every $\gamma_{j,l}$. Together with the uniqueness of the Plateau disk \cite{Lawson, KaW}, this shows  the theorem of Kapouleas and Wiygul \cite{KaW2} on the uniqueness of the Lawson minimal surface $\xi_{m,k}$ holds for  smaller symmetry groups.

\begin{proof} [Proof of  Theorem \ref{th}] For the case $mk=1$, Theorem \ref{th} holds thanks to the solution of the Willmore conjecture by Marques and Neves \cite{Marques0},
so we may assume $mk>1$. Since $G_{m,k}$ is a subgroup of each of $\tg$, ${G}^P_{m,k}$ and ${G}^Q_{m,k}$,
Theorem \ref{thm-inter} implies  that $\pi(M)$ either contains $\mathbb{L}$ or meets  $\mathbb{L}$ at two points.

If $\pi(M)$ meets $\mathbb{L}$ at two points, consider the curve $c^{*}\subset\pi(M)\cap(\mathbb{S}_1\bowtie\mathbb{L}^*)$ introduced in the proof of Theorem \ref{thm-inter}.

For the case of $\tg$, the additional $\underline{\gamma_{1,1}^*}$-symmetry induces $\underline{\mathbb{L}^*}$-symmetry in  $\s/\CC$.
If $\V\in \pi(M) \cap\mathbb{S}_1$ is different from $\V^{*\pm}_1$, then $\V$, $\underline{\mathbb{L}}(\V)$, $\underline{\mathbb{L}^*}(\V)$ and  $\underline{\mathbb{L}}\cdot \underline{\mathbb{L}^*}(\V)$ give four different  points in $\pi(M) \cap\mathbb{S}_1$, contrary to $v_1=2$. So  $c^{*} \cap\mathbb{S}_1=\pi(M) \cap\mathbb{S}_1=\{\V^{*\pm}_1\}$.
Since $G^*_{m,k}\subset \tg$, by Lemma \ref{Le-PQ}, $\mathbb{L}^*\subset\pi(M)$.
So after a rotation of $\s$ taking  $\{\gamma^*_{j,l}\}$ to $\{\gamma_{j,l}\}$, $\mathbb{L}\subset\pi(M)$ always holds and we can assume $M$ contains all of $\{\gamma_{j,l}\}$ without loss of generality.

For the case of ${G}^P_{m,k}$,  all of $\hp(M)$, $\mathbb{S}_1$ and $\mathbb{S}_1\bowtie\mathbb{L}^*$ admit the additional $\F_1^h$-symmetry. If $\pi(M)$ were to meet $\mathbb{L}$ at two points, then  the loop  $c^{*}$  is $\F_1^h$-symmetric because it contains both points of $\mathbb{S}_1\cap\hp(M)$. So $\hp(M)$ is divided by $c^*$ into two disks: $D$ and $\F_1^h(D)$. Note that $\mathbb{S}_1\bowtie\mathbb{L}^*$ separates $\pi(\s)$ into two $\F_1^h$-invariant domains. Hence, a small collar neighborhood of $c^{*}=\partial D$ in $D$ contained in one of these domains yields a collar neighborhood of $c^{*}=\partial \F_1^h(D)$ in $\F_1^h(D)$ contained in the same domain. This contradicts  transversality of    $\hp(M)\cap\mathbb{S}_1\bowtie\mathbb{L}^*=c^{*}$.
The case of ${G}^Q_{m,k}$ is similar, so we always have $\mathbb{L}\subset\pi(M)$.

Therefore we conclude that $M$ is M\"{o}bius congruent to a surface $\tilde{M}$ which contains every $\gamma_{j,l}$.
Note that $\tilde{M}$ can be constructed from a disk $\tilde{M}_{0,0}$ with boundary $\Gamma_{0,0}$ under action of halfturns about $\gamma_{j,l}$. By the construction \cite{Lawson} of $\xi_{m,k}$, the disk $\delta_{0,0}=\mathcal{T}_{0,0}\cap\xi_{m,k}$ is the unique (see \cite{KaW}) solution to the Plateau problem for surfaces with boundary $\Gamma_{0,0}$, and so \[\area(\tilde{M}_{0,0})\geq \area(\delta_{0,0}).\]
Summing this inequality over the orbit of $\g$ we obtain
	\[W(M)=W(\tilde{M})\geq \area(\tilde{M})\geq \area(\xi_{m,k})=W(\xi_{m,k}),\]
	and equality holds if and only if $M$ is congruent to $\xi_{m,k}$.
Thus the first inequality in Theorem \ref{th} holds. The second inequality in Theorem \ref{th} comes from \cite{KW3} (see Remark \ref{remark} below).
\end{proof}
\begin{remark}\label{remark}
To prove the second inequality in Theorem \ref{th}, we observe \cite{KW3} (see also \cite{SchoenJDG83,KaW,KaW2})
that the reflection symmetry interchanging $P_0$ and $P_1$ (or the reflection interchanging $Q_0$ and $Q_1$, respectively)
decomposes $\dd_{0,0}$ into 4 congruent {\em graphical} pieces, that is, each piece projects injectively via rotation about $Q_0Q_1$
(or about $P_0P_1$, respectively) into the great-sphere triangle $P_0Q_0Q_1$ (or into the great-sphere triangle $P_0P_1Q_0$); moreover,
these roto-projections decrease area, so a simple interpolation argument \cite{KW3} also yields an effective {\em lower} bound
on $\area(\dd_{0,0})$ and thus on $\area(\xi_{m,k})$: using these lower bounds and the earlier upper bound
$\area(\xi_{m,k})<4\pi(k+1)$, when $m\geq k >1$, we find \cite{KW3} that $\area(\xi_{m,k})>\area(\xi_{mk,1})$.
\end{remark}

Until now, we have been interested in Willmore minimizers. Recently, Kapouleas and Wiygul \cite{KaW2} showed that a closed, embedded minimal surface in $\s$ of genus $mk$ is congruent to $\xi_{m,k}$, if its symmetry group is conjugate in $O(4)$ to $\hg$, the full symmetry group of $\xi_{m,k}$ when $m>k$.  As a by-product, we show that their result holds under relaxed symmetry and genus assumptions:
\begin{theorem}\label{th-kw}
	Let $N\subset\s$ be a closed, embedded minimal surface of genus $g$ with $1<g\leq mk+k$ which is symmetric under a group $G$ conjugate in $O(4)$ to either $\widetilde{G}_{m,k}$, ${G}^P_{m,k}$ or ${G}^Q_{m,k}$.  Then $N$ has genus $g=mk$ and is congruent to $\xi_{m,k}$.
 Equivalently,  in the orbifold $\s/\CC$, the orbifold surface $\pi(\xi_{m,k})$ is the only minimal sphere invariant under any of the three groups up to the isometries of  $\s/\CC$:
 \[\begin{split}
 &    \widetilde{G}_{m,k}/\CC=\la\F_1^v\cdot\F_2^v,\F_1^h\cdot\F_2^h\ra ,\\
 &{G}^P_{m,k}/\CC=\la\F_1^v\cdot\F_2^v, \mathcal{F}^h_1 \ra,\\
 &{G}^Q_{m,k}/\CC=\la\F_1^v\cdot\F_2^v,\mathcal{F}^h_2 \ra.
 \end{split}\]
\end{theorem}

\begin{proof}
Recall \eqref{eq-G-seq} the index-$2$ subgroup $G_{m,k}$  of the three groups $\widetilde{G}_{m,k}$, ${G}^P_{m,k}$ and ${G}^Q_{m,k}$. Theorem \ref{thm-inter} and the proof of Theorem \ref{th} show (up to isometry of $\s$) that the minimal surface $N$ is the $G_{m,k}$-orbit of a minimal disk $N_{0,0}$ bounded by $\Gamma_{0,0}$. We claim $N_{0,0}=N\cap \TT_{0,0}$. This will imply $N_{0,0}=\delta_{0,0}$ is the unique such disk \cite{Lawson, KaW} and hence  $N=\xi_{m,k}$.

To see $N_{0,0}=N\cap \TT_{0,0}$, we need only to show that $0\leq\theta\leq\frac{\pi}{k+1}$ and  $0\leq\theta^\perp\leq\frac{\pi}{m+1}$ on $N_{0,0}\cap \s_\gamma$, 
 where we extend the parameters $\theta$ on $\gamma$ and $\theta^\perp$ on $\gamma^\perp$ from \eqref{eq:theta} to
  \[\begin{split}
      \s_\gamma&:=\s\backslash(\gamma\cup\gamma^\perp)=\left\{(\cos te^{i\theta^{\perp}},\sin te^{i\theta}) |\theta,\theta^{\perp}\in [-\pi,\pi), t\in (0,\frac{\pi}{2}) \right\}\\
      &\cong S^1\times S^1\times(0,\frac{\pi}{2}).
  \end{split}\]
  Since $N$ is embedded, for small enough $\epsilon>0$, there is a collar neighborhood $ C_{\epsilon}$ of $\partial N_{0,0}$ such that  $-\epsilon<\theta<\frac{\pi}{k+1}+\epsilon$ and  $-\epsilon<\theta^\perp<\frac{\pi}{m+1}+\epsilon$ on $ C_{\epsilon}\setminus\partial N_{0,0}$. By Theorem \ref{thm-inter}, $N$ meets $\gamma\cup\gamma^\perp$ only at $P_j$ and $Q_l$, so the complement $N_{\epsilon}=N_{0,0}\setminus C_{\epsilon}$ is a compact minimal disk in $\s_\gamma$.
 The parameters $\theta$ and $\theta^\perp$  lift to $\R$-valued angle functions $\Theta$ and $\Theta^\perp$ on the universal cover $\widetilde{\s_\gamma}\cong\mathbb{R}^3$. 
 Since the ranges of $\theta$ and $\theta^\perp$ in $\partial N_{\epsilon}$ are smaller than $\pi$, the preimage of $\partial N_{\epsilon}$ in $\widetilde{\s_\gamma}$ consists of disjoint loops. So the preimage of $N_{\epsilon}$  in $\widetilde{\s_\gamma}$ consists of disjoint compact minimal disks,  on each of which $\Theta$ and $\Theta^\perp$ are bounded. By the maximum principle, extreme values of $\Theta$ and $\Theta^\perp$ are attained at each boundary, which leads to $-\epsilon\leq\theta\leq\frac{\pi}{k+1}+\epsilon$ and $-\epsilon\leq\theta^\perp\leq\frac{\pi}{m+1}+\epsilon$ on $N_\epsilon$. This holds for arbitrarily small $\epsilon$, so $N_{0,0}\subset\TT_{0,0}$. 
\end{proof}

\begin{remark}\
\begin{enumerate}
    \item Examining Figure \ref{Tmk}, if $P_0^*$ and $ \pi^{-1}(\mathbb{A}^{*+})$ moved close to $P_0$, and if $Q_0^*$ and $ \pi^{-1}(\mathbb{A}^{*-})$ moved close to $Q_1$,
then $M$ would be made of tubes of small area but possibly large  squared mean curvature integral:  competition between these two terms in the Willmore energy $W(M)$ makes
reducing the symmetry group to $G_{m,k}$ in Theorem \ref{th-kw} and Theorem \ref{th} more subtle.

\item Concerning the uniqueness question for a  $G_{m,k}$-symmetric minimal surface $M$ of genus $mk$, by Theorem \ref{thm-inter},
either the surface  contains $\cup\gamma_{j,l}$, in which case the uniqueness holds true \cite{Lawson}, \cite{KaW2}, or the surface intersects $\gamma$, $\gamma^{\perp}$ and $\cup\gamma_{j,l}$ orthogonally (see Figure \ref{Tmk}). Hence, in the latter case, the uniqueness of $\xi^*_{m,k}$ depends on the classification of embedded free boundary minimal disks in the tile $\mathcal{T}_{0,0}$ missing a pair of opposite edges. Note that the Clifford torus provides an example missing the pair $P_0P_1$ and $Q_0Q_1$ (see Figure \ref{T1}) which is different from the case of
$\xi^*_{m,k}$. For further discussion, see Remark 1.3 of \cite{Wiygul} and Remark 4.6-4.7 of \cite{KaW2}.

\item When $m\to\infty$ for fixed $k$, the Lawson surfaces $\xi_{m,k}$ converge as varifolds to the union of $k+1$ equally-spaced great $2$-spheres meeting in the common great circle 
$\gamma^\perp$. Rescaling at a point $Q_0\in\gamma^\perp\subset\s\subset\R^4$
such that the rescaled copies of $\xi_{m,k}$ have maximum curvature $1$, they converge smoothly on compact subsets of 
$\s\setminus\{-Q_0\}$ to a \emph{Karcher tower} in $\R^3$ (identfied with the rescaled tangent space 
$T_{Q_0}\s$) asymptotic to $2k+2$ equally-spaced half-planes (sometimes called \emph{wings}) meeting in their common boundary line (the rescaling limit of the great circle $\gamma^\perp$).  These equiangular Karcher towers generalize the classical Scherk tower (the case $k=1$) whose 4 wings meet at right angles like a {\Large${\times}$} sign. The wings of the Scherk tower can be ``flapped" to meet in alternating angles like an {\small $/\!\!\!\backslash$}, and analogous wing-flapping of the Karcher towers is possible for $k\ge2$.
\item
An interesting open question (related to progress of Kapouleas \cite[Section 4]{Kap} on desingularizing intersecting minimal surfaces) is whether the hemispherical wings of the Lawson surfaces also can be flapped? Symmetry under the binary cyclic group $\CC$ allows for this, as does symmetry under other subgroups of $\hg$ in \eqref{eq-groups} containing $\CC$: the halfturn group $\g^*$ in $SO(4)$, as well as the reflection groups $\CC^P$, $\CC^Q$ and $\hCC=\check{G}_{m,k}$ in $O(4)$.
\item
In fact, Kapouleas and Wiygul  \cite{KaW2} ruled out  ``infinitesimal wing flapping" for the case $k=1$, and the moduli space theory developed in \cite{KMP} implies these surfaces are isolated. This leaves open the possibility that the hemispherical wings could ``jump" by some positive angle. One way to rule this out for  $G^*_{m,k}$ or the reflection group $\hCC$  would  also be to classify embedded free boundary minimal disks in the tile $\TT^*_{0,0}$ missing two opposite edges, as discussed in (2) for $G_{m,k}$ (since $G^*_{m,k}$ is conjugate to  $G_{m,k}$).
\item We do not study the index-$2$ subgroup 
\[
\mathrm{Diff}^+(\xi_{m,k})\cap\hg=\la G^*_{m,k}, \underline{\gamma_{0,0}}\cdot\underline{\Sigma_{P_0^*}}\ra
\]
preserving the surface orientation, since  a surface of the same genus with this symmetry group need not satisfy the conclusion of Theorem \ref{thm-inter}, nor does it allow  wing-flapping. For the same reason we do not study the remaining index-$2$ subgroups of $\hg$:
\[
\la R_{m,k},\underline{\Sigma_{Q_0}},\underline{\gamma_{0,0}^*}\cdot\underline{\Sigma_{P_0}}\ra,\quad\la R_{m,k},\underline{\Sigma_{P_0}},\underline{\gamma_{0,0}^*}\cdot\underline{\Sigma_{Q_0}}\ra.
\]
\end{enumerate}
\end{remark}

\section{$\g$-equivariant isotopy classes of embedded surfaces in $\s$}

In this section we will derive a more precise description of the isotopy class of a $\g$-symmetric embedded surface $M$ based on the topology of $\pi(M)$. An isotopy of $M$ in $\s$ is called $\g$-equivariant if it preserves the $\g$-symmetry. 
Now we state the main result of this section.
\begin{theorem} \label{th-G-inv}
	Let $M$ be a closed, embedded, $\g$-symmetric surface in $\s$ of genus $g$. If $1<g\leq (m+1)k$, then $g=mk$ and $M$ is $\g$-equivariantly isotopic to $\xi_{m,k}$.
\end{theorem}

We prove Theorem \ref{th-G-inv} in  $\pi(\s)=\s/\CC$, where $\pi(M)$ is an $\underline{\mathbb{L}}$-symmetric genus $0$ surface.

An isotopy of $\pi(\s)$ is called \emph{admissible} if it restricts to an isotopy on each of  $\mathbb{S}_1$, $\mathbb{S}_2$ and $\mathbb{L}$,
and  also commutes with the halfturn symmetry $\underline{\mathbb{L}}$ (such an isotopy respects the intersection patterns of $\pi(M)$ with $\mathbb{S}_1$, $\mathbb{S}_2$ and $\mathbb{L}$). Hence an admissible isotopy lifts to a $\g$-equivariant isotopy of $M$ in $\s$. In our proof, the points on $\mathbb{S}_1$, $\mathbb{S}_2$ and $\mathbb{L}$ are called {\em restricted points}.

By Theorem \ref{thm-inter}, $\pi(M)$ has two intersection patterns with $\mathbb{L}$.
So Theorem \ref{th-G-inv} is a corollary of the following proposition.

\begin{proposition}\label{prop-iso-1} \
\begin{enumerate}
    \item[(I)] When $\mathbb{L}\subset\hp(M)$, $\hp(M)$ is admissibly isotopic to $\hp(\xi_{m,k})$ or $\hp(\xi^o_{m,k})$.
    \item[(II)] When $\mathbb{L}\not\subset\hp(M)$, $\hp(M)$ is admissibly isotopic to $\pi(\xi^*_{m,k})$ or $\pi(\xi^{*o}_{m,k})$.
\end{enumerate}
\end{proposition}

A key idea of the proof is to simplify   intersection patterns of $\pi(M)$ with some specific  2-spheres \begin{equation}
    \mathcal{S}_i:=\mathbb{S}_i\bowtie\mathbb{L}\hbox{ in both cases and also } \mathcal{S}_i^*:=\mathbb{S}_i\bowtie\mathbb{L}^* \hbox{ in Case (II)}
\end{equation} by admissible isotopies, so that $\pi(M)$ is composed of $\mathbb{L}$ or $\mathbb{L}^*$ and two disks in the interior of a pair of opposite lenses $\pi(\TT_{i,j})$ or  $\pi(\TT^*_{i,j})$ in $\s/\CC$, where $\TT_{i,j}^*=P^*_{i}P^*_{i+1}\bowtie Q^*_{j}Q^*_{j+1}$ (see the right-hand sides of Figures \ref{Fig:lawson} and \ref{Tmk}).

\begin{proof}
All our configurations and isotopies are $\underline{\mathbb{L}}$-symmetric in $\pi(\s)\cong\s$.
Up to  an admissible isotopy, we can assume $\pi(M)$ meets a given $\underline{\mathbb{L}}$-symmetric 2-sphere $\mathcal{S}$ transversely in disjoint simple loops. An innermost-disk argument using finger-moves simplifies  $\pi(M)\cap\mathcal{S}$: any loop bounding some open disk $D\subset\pi(M)$ such that $\overline{D}\cap (\mathbb{S}_1\cup \mathbb{S}_2\cup \mathbb{L})=\emptyset$, can be eliminated by an admissible isotopy of $\pi(M)$. So we will assume there exist no such loops or disks.

For (I), we have $\mathbb{L}$ separating $\pi(M)$ into two disks $D_1$ and $D_2$ since $\hp(M)$ has genus $0$ by Theorem \ref{thm-inter}. The $2$-sphere
$\mathcal{S}_1$ separates  $D_1$ into several connected components, at least one of which is a disk with restricted points in its closure. Since $\mathbb{L}\cap \mathbb{S}_i=\mathbb{V}_i^{\pm}$ and $v_i=2$, all the restricted points on $\pi(M)$ are contained in  $\mathbb{L}$, so ${D}_1$ itself is the only connected component, and therefore ${D}_1\cap \mathcal{S}_1=\emptyset$.  Similarly, $D_1\cap \mathcal{S}_2=\emptyset$.
So $D_1$ is completely contained in some lens $\pi(\TT_{i,j})$, and $\overline{D_1}$ is admissibly isotopic to $\pi(\xi_{m,k})\cap \pi(\TT_{i,j})$ or $\pi(\xi_{m,k}^o)\cap \pi(\TT_{i,j})$, thus $\pi(M)$ is admissibly isotopic to $\pi(\xi_{m,k})$ or $\pi(\xi_{m,k}^o)$ respectively, since $D_2=\underline{\mathbb{L}}(D_1)$.

For (II),  $\hp(M)$ meets $\mathbb{L}$ transversely at two points $C^+,C^-$.
Since $v_i=2$ by Theorem \ref{thm-inter},
we may perform an admissible isotopy of $\pi(M)$ such that $\{\mathbb{V}_1^{*\pm}, \mathbb{V}_2^{*\pm}\}\subset\pi(M)$, as in Figure \ref{Tmk}.

Consider the sphere $\mathcal{S}_1$ which intersects $\pi(M)$ in disjoint simple loops.
Since $\mathbb{S}_1$ meets $\pi(M)$ transversally at exactly two points $\V_1^{*\pm}$,
there exists a unique simple loop $c\subset \mathcal{S}_1\cap\pi(M)$ containing  $\V_1^{*\pm}$. 
Hence $\underline{\mathbb{L}}(c)=c$.
Thus \[c\cap \mathbb{L}=\pi(M)\cap \mathbb{L}=\{C^+,C^-\}\] and $c$ separates $\pi(M)$ into two disjoint open disks, $D_1$ and $D_2$, each containing one of $\V_2^{*\pm}$.
Consider the connected components of $D_1$ and $D_2$ separated by $\mathcal{S}_1$.
At least one of the components on each $D_i$ is a disk $D_{i}'$ with restricted points in its closure, so that   $\overline{D_{i}'}$   contains either $c$ or $\V_2^{*\pm}$.
If $D_{i}'$ contains $\V_2^{*+}$ or $\V_2^{*-}$--- the transverse intersection points of $\mathbb{S}_2$ with $\pi(M)$ --- then $\partial D_{i}'\subset\mathcal{S}_1$ is linked with $\mathbb{S}_2$. So  $\partial D_{i}'$ meets $\mathbb{L}$ at $C^\pm$, that is, $\partial D_i'=c$. This implies $D_i'=D_i$ and $\pi(M)\cap \mathcal{S}_1=c$. Similarly, $\mathcal{S}_2$ also separates $\pi(M)$ into two disks. Thus $\pi(M)$ is separated by $\mathcal{S}_1$ and $\mathcal{S}_2$ into four disks, one contained in each of four lenses $\pi(\TT_{0,0})$, $\pi(\TT_{1,1})$, $\pi(\TT_{0,1})$ and  $\pi(\TT_{1,0})$ tiling $\pi(\s)$. In each lens, the restricted points are all located on the boundary. So on the disk in each lens we can find a simple path avoiding restricted points and connecting one of $\V_1^{*\pm}$ to one of $\V_2^{*\pm}$ such that these four paths form an $\underline{\mathbb{L}}$-symmetric simple loop on $\pi(M)$. An admissible isotopy respecting the four lenses takes this loop to $\mathbb{L}^*$.

Now  consider the connected components of $\pi(M)$ separated by $\mathcal{S}_i^*$.
As we argued above to show $\pi(M)\cap \mathcal{S}_1=c$, similarly we get $\pi(M)\cap\mathcal{S}_i^*=\mathbb{L}^*$ using $C^+$, $C^-$ and the circle  $\mathbb{L}^*$ in place of $\V_2^{*+}$, $\V_2^{*-}$ and the loop $c$. So $\mathbb{L}^*$ separates $\pi(M)$ into two disjoint disks, neither meeting
$\mathcal{S}_i^*$.
Because $\pi(M)$ meets $\mathbb{S}_i$ transversely at $\V_i^{*\pm}$, these two disks must lie in a pair of opposite lenses: either $\pi(\TT^*_{0,0})$ and $\pi(\TT^*_{1,1})$, or $\pi(\TT^*_{0,1})$ and  $\pi(\TT^*_{1,0})$; in the first case,
$\pi(M)$ is admissibly isotopic to $\pi(\xi^*_{m,k})$, and in the second, to $\pi(\xi^{*o}_{m,k})$.
\end{proof}

\vspace{2mm}
\begin{remark}\label{rmk-t2}
\
\begin{enumerate}
    \item
 The condition $g>1$ is necessary; in fact, there are two types of $G_{m,k}$-symmetric tori which are obviously not even homeomorphic to $\xi_{m,k}$ when $mk>1$:
	\begin{enumerate}
		\item The homogeneous torus $\{(ae^{i\phi},be^{i\psi}): a^2+b^2=1, \phi ,\psi\in [0,2{\pi}]\}$ is  $G_{m,k}$-symmetric and projects (Figure \ref{T1}) to a torus in $\pi(\s)$.
		\begin{figure}[htbp]
			\includegraphics[width=.8\textwidth]{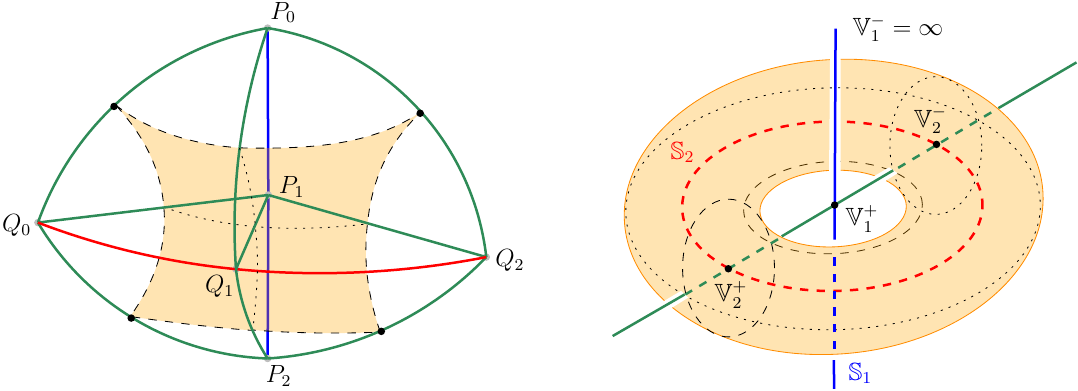}
			\caption{$G_{m,k}$-symmetric homogeneous torus in $\DD$ and its projection to $\s/\CC$}\label{T1}
		\end{figure}
		\item
		If $\hbox{gcd}(m+1,k+1)=1$, a knotted $G_{m,k}$-symmetric torus  exists  (see Figure \ref{T2}), in which case $\CC$ admits an element of order $(m+1)(k+1)$.
				\begin{figure}[htbp]
			\includegraphics[width=.75\textwidth]{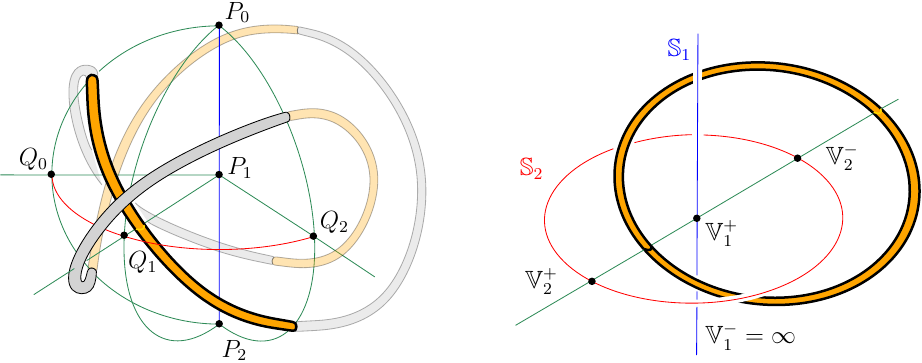}
			\caption{$G_{2,1}$-symmetric  torus in $\DD$ and its projection to  $\s/R_{2,1}$}\label{T2}
		\end{figure}	
	\end{enumerate}
 \vspace{2mm}
\item The upper bound on the genus in Theorem \ref{th-G-inv} is sharp since there exists $G_{m,k}$-symmetric surfaces (see Figure \ref{Fig:T3}) with genus $mk+k+1$.
\begin{figure}[htbp]
	\includegraphics[width=0.7
 \textwidth]{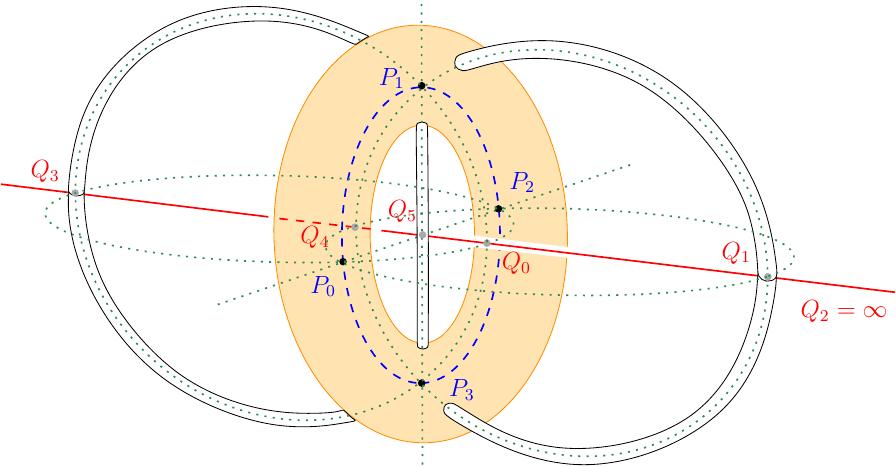}
	\caption{The surface  constructed from the homogeneous torus by adding handles surrounding part of $\gamma_{i,j}$ ($i,j$ are odd). Here $m=2,k=1,g=4$.}\label{Fig:T3}
\end{figure}
\item
  The group $\g$ in Theorem \ref{th-G-inv} is also optimal in the sense that there exist infinitely many knotted genus-$mk$ surfaces with $\CC$-symmetry (see Figure \ref{FD2}) which are homeomorphic but not isotopic to any embedded minimal surface \cite{Lawson2}.
    \begin{figure}[htbp]
			\includegraphics[width=.8\textwidth]{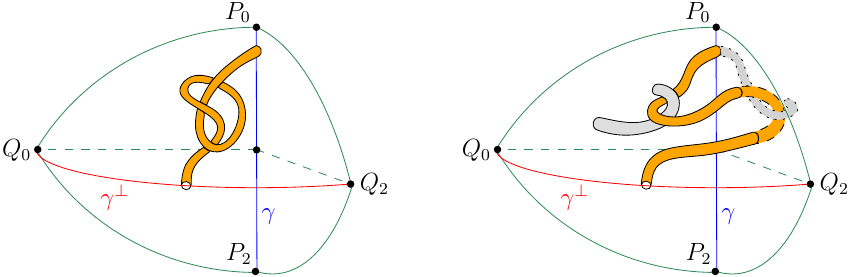}
 			\caption{These two types of surfaces and their combinations can create many surfaces homeomorphic but not isotopic to $\xi_{m,k}$.}\label{FD2}
		\end{figure}
\end{enumerate}
\end{remark}
\vspace{2mm}{\bf Acknowledgements}:  We are grateful to E.~Kuwert, X.~Ma, P.~McGrath, R.~Schwartz, C.~Song, C.~Xia, Z.X.~Xie, and particularly to C.P.~Wang, for several valuable discussions during this project; special appreciation is due to K.~Brakke for producing Figure \ref{Fig:SdC}.  An earlier proof of Theorem  \ref{th} for surfaces with the smaller symmetry group $G=G_{m,k}$ had overlooked part of the singular set for the orbifold $\s/G$: we thank N.~Kapouleas and D.~Wiygul for pointing out $\xi^*_{m,k}$ would have been a counterexample to one of our earlier lemmas, and also for encouraging us to address the ``wing flapping" question, as we discuss briefly in Section 4. RK was supported in part by KIMS and Coronavirus University during the pandemic, as well as by National Science Foundation grant DMS-1928930 while at the Simons Laufer Mathematical Sciences Institute (formerly MSRI) in Berkeley, California, where this work was completed during the Fall 2024 term; YL was supported by NSFC Project 12171473; PW was supported by NSFC Project 12371052 and 11971107.

\def\refname{References}

\end{document}